\documentclass[11pt,letterpaper,reqno]{amsart}
\usepackage{tikz}
\usetikzlibrary{positioning, shapes.geometric, arrows.meta}
\usepackage{amssymb}
\usepackage{amsmath}
\usepackage{amsthm}
\usepackage{amsfonts}
\usepackage{bbm}
\usepackage{enumitem} 
\usepackage{pgfplots}
\pgfplotsset{compat=1.18} 
\usepackage{booktabs}

\usepackage{graphicx}
\usepackage[T1]{fontenc}
\usepackage{doi}
\usepackage{float} 
\usepackage{tikz}
\usepackage{pgfplots}
\pgfplotsset{compat=1.18}
\usetikzlibrary{arrows.meta}

\hypersetup{
	colorlinks=true,
	linkcolor=red,       
	citecolor=green,     
	urlcolor=magenta,    
	filecolor=magenta,
	linktoc=all,
	pdfstartview={FitH}
}

\newtheorem{thm}{Theorem}[section]
\newtheorem{lem}[thm]{Lemma}
\newtheorem{prop}[thm]{Proposition}
\newtheorem{cor}[thm]{Corollary}

\theoremstyle{definition}

\newtheorem{defin}[thm]{Definition}

\theoremstyle{remark}
\newtheorem{rem}[thm]{Remark}
\numberwithin{equation}{section}

\makeatother

\newcommand{\R}{\mathbb{R}} \newcommand{\E}{\mathbb{E}} \newcommand{\Cov}{\operatorname{Cov}} \newcommand{\tr}{\operatorname{tr}} \newcommand{\one}{\mathbf{1}} \newcommand{\norm}[1]{\left\lVert #1\right\rVert}
\newcommand{\cK}{\mathcal{K}}
\newcommand{\rank}{\operatorname{rank}}
\newcommand{\diag}{\operatorname{diag}}
\newcommand{\osc}{\operatorname{osc}}
\newcommand{\adj}{\operatorname{adj}}
\newcommand{\Span}{\operatorname{span}}
\allowdisplaybreaks

\begin{document}

\title[The \(A\)-optimal design problem and the \(S\)-matrix conjecture]
{The \(A\)-optimal design problem and the \(S\)-matrix conjecture}

\author[T.~Zhang]{Teng Zhang}

\address{School of Mathematics and Statistics, Xi'an Jiaotong University, Xi'an 710049, P. R. China}
\email{teng.zhang@stu.xjtu.edu.cn}



\begin{abstract}
	The \(S\)-matrix conjecture was formulated by Sloane and Harwit in
	1976, motivated by an \(A\)-optimal design problem arising in
	spectroscopy. It states that if \(A\) is a nonsingular real
	\(n\times n\) matrix whose entries lie in \([0,1]\), then
	\[
	\left\|A^{-1}\right\|_F
	\ge
	\frac{2n}{n+1}.
	\]
	Moreover, equality holds if and only if \(A\) is an \(S\)-matrix. In this paper, we prove the conjecture by combining a variational trace inequality in odd dimensions with a sharp range-constrained estimate for the Moore--Penrose inverse of a centered matrix in even dimensions.
\end{abstract}

\maketitle

\section{Introduction} \label{sec:introduction} 

Optimal experimental design concerns the choice of observations that estimate unknown parameters as accurately as possible under structural, physical, or economic constraints. In a linear model, the quality of a design is encoded by its information matrix, and different scalar summaries of the inverse information matrix lead to different optimality criteria. Among the classical criteria is $A$-optimality, which minimizes the sum of the variances of the estimated parameters.

Consider the saturated linear model 
\begin{equation*} 
	Y=A\theta+\varepsilon, \qquad \E(\varepsilon)=0, \qquad \Cov(\varepsilon)=\sigma^2 I_n,
\end{equation*} where \(A\in\R^{n\times n}\) is nonsingular and \(\theta\in\R^n\) is unknown. The least-squares estimator is 
$\widehat{\theta}=A^{-1}Y$, and \[ \Cov(\widehat{\theta}) =\sigma^2 A^{-1}A^{-\top} =\sigma^2(A^\top A)^{-1}. \] Consequently, 
\begin{equation} 
	\E\norm{\widehat{\theta}-\theta}_2^2 =\sigma^2\tr\bigl((A^\top A)^{-1}\bigr) =\sigma^2\norm{A^{-1}}_{F}^2. \label{eq:Acriterion} 
\end{equation}
Thus, in the saturated setting, minimizing the Frobenius norm of \(A^{-1}\) is exactly the \(A\)-optimal design problem. Equivalently, it minimizes the total mean squared error of the least-squares estimator; after division by \(n\), it minimizes the average variance of its coordinates.

For \(A=(a_{ij}),B=(b_{ij})\in\mathbb{R}^{n\times n}\), the Frobenius inner product is defined by
\[
\langle A,B\rangle
:=
\operatorname{tr}(B^\top A)=\sum_{i=1}^n\sum_{j=1}^na_{ij}b_{ij}.
\]
The induced Frobenius norm is
$
\lVert A\rVert_{F}
:=
\sqrt{\langle A,A\rangle}
=
\left(
\sum_{i=1}^{n}\sum_{j=1}^{n}a_{ij}^{2}
\right)^{1/2}.
$

The signed design problem originates in Hotelling's classical work on
chemical-balance weighing experiments in 1944 \cite{Hot44}. In such an experiment,
an object may be placed on either side of the balance, so the entries of the
design matrix \(A=(a_{ij})\) naturally take both positive and negative values.
More generally, suppose that
\[
-1\le a_{ij}\le  1,
\qquad 1\leq i,j\leq n.
\]
The Cauchy--Schwarz inequality gives
\[
n^2
=
\bigl\langle A^\top,A^{-1}\bigr\rangle^{2}
\leq
\lVert A^\top\rVert_{F}^{2}
\lVert A^{-1}\rVert_{F}^{2}
\leq
n^2\lVert A^{-1}\rVert_{F}^{2}.
\]
Consequently,
\begin{equation}
	\lVert A^{-1}\rVert_{F}\geq 1.
	\label{eq:signedbound}
\end{equation}
Equality holds if and only if \(A\) is a \emph{Hadamard matrix}, that is,
$
A\in\{-1,1\}^{n\times n},
$ and $
A^\top A=nI_n.
$
In this case,
\[
\Cov(\widehat{\theta})
=
\frac{\sigma^2}{n}I_n,
\]
so all parameters are estimated with the same variance, and their estimation
errors are uncorrelated. Thus, Hadamard matrices provide optimal signed
weighing designs; see also \cite{HS79}.

The admissible entries of a design matrix, however, depend on the underlying
experimental mechanism. In coded spectroscopy, an aperture or spectral channel
is typically closed, open, or partially transmitting. Negative transmission
coefficients are therefore not physically available, and the natural
constraint becomes
\begin{equation}
	0\leq a_{ij}\leq 1,\qquad 1\leq i,j\leq n.
	\label{eq:box}
\end{equation}
This nonnegativity constraint fundamentally changes the design problem:
For \(n>1\), this nonnegativity constraint excludes Hadamard matrices;
indeed, no nonnegative matrix can attain equality in the signed bound
\eqref{eq:signedbound}. Determining the sharp
lower bound for \(\lVert A^{-1}\rVert_{F}\) under
\eqref{eq:box} leads to the \(S\)-matrix conjecture.

The corresponding problem under the nonnegativity constraint \eqref{eq:box} arose in the study of masks for Hadamard-transform spectrometers and image scanners. It was formulated by Sloane and Harwit \cite{SH76} and developed
further in \cite{HS79,Slo79}. The candidate extremizers are obtained from Hadamard matrices. Let \(\one\in\mathbb{R}^n\) be the all-ones vector.
Let \(H\) be a normalized Hadamard matrix of order \(n+1\), written in
the form
\[
H=
\begin{pmatrix}
	1 & \one^{\top}\\
	\one & C
\end{pmatrix},
\qquad
C\in\{-1,1\}^{n\times n}.
\]
The matrix
\[
S=\frac12\bigl(\one\one^{\top}-C\bigr)
\]
is a \(0\)--\(1\) matrix and is called an \emph{\(S\)-matrix}. Clearly, the existence of an \(S\)-matrix of order \(n\) is
equivalent to the existence of a Hadamard matrix of order \(n+1\). Since a
Hadamard matrix of order \(N>2\) can exist only when \(N\) is divisible
by \(4\), a nontrivial \(S\)-matrix can exist only when
\[
n\equiv3\pmod4;
\]
the exceptional case \(n=1\) corresponds to the Hadamard matrix of
order \(2\). The famous Hadamard conjecture \cite[Conjecture~1]{Zha08} asserts that a Hadamard matrix of
order \(4m\) exists for every positive integer \(m\), or equivalently,
that an \(S\)-matrix of order \(4m-1\) exists for every positive
integer \(m\). Hadamard matrices are currently known to exist for every
admissible order at most \(664\), while \(668\) is the smallest
unresolved order. More specifically, among the admissible orders not
exceeding \(1208\), the only unresolved cases are
\[
668,\qquad716,\qquad892,\qquad1132;
\]
see \cite{CP25}.

Equivalently, a matrix \(S\in\{0,1\}^{n\times n}\) is an
\(S\)-matrix if and only if it satisfies either, and hence both, of
the Gram identities
\begin{equation}
	S^{\top}S
	=
	\frac{n+1}{4}
	\bigl(I_n+\one\one^{\top}\bigr),
	\qquad\text{or equivalently}\qquad
	SS^{\top}
	=
	\frac{n+1}{4}
	\bigl(I_n+\one\one^{\top}\bigr).
	\label{eq:Sgram}
\end{equation}
This characterization is standard in the theory of Hadamard designs;
see \cite[Sec.~V.1]{CK07}. The two Gram identities are also recorded
explicitly in \cite[p.~51]{KM08}.

Since 
$
\bigl(I_n+\one\one^\top\bigr)^{-1} =I_n-\frac{1}{n+1}\one\one^\top, 
$ we obtain 
\begin{equation} 
	\norm{S^{-1}}_{F}^2 =\tr\bigl((S^\top S)^{-1}\bigr) =\frac{4n^2}{(n+1)^2}, \qquad \norm{S^{-1}}_{F}=\frac{2n}{n+1}. \label{eq:Svalue} 
\end{equation}

Motivated by this construction, Sloane and Harwit \cite{SH76} conjectured that the value in \eqref{eq:Svalue} is the universal lower bound for every nonsingular matrix satisfying \eqref{eq:box}, and that equality characterizes the \(S\)-matrices. This assertion is now known as the \emph{\(S\)-matrix conjecture},  which was explicitly recorded 
in Zhan's remarkable problem collection, ``\emph{Open Problems in Matrix Theory}'', as \cite[Conjecture~11]{Zha08}. The main result of this paper is to prove this conjecture.

\begin{thm}[$S$-matrix conjecture]\label{thm:main} Let  \(A\in\R^{n\times n}\) be nonsingular with 
	$ 0\leq a_{ij}\leq 1,$ $1\leq i,j\leq n. $
	Then \begin{equation} 
		\norm{A^{-1}}_{F}\geq \frac{2n}{n+1}. \label{eq:mainbound} \end{equation} Equality holds if and only if \(A\) is an \(S\)-matrix. 
\end{thm}
The \(S\)-matrix conjecture has a long history in optimal design and matrix analysis,
which may be
summarized chronologically as follows:
\begin{itemize}
	\item In 1960, Kiefer and Wolfowitz \cite{KW60} established an
	equivalence theorem for optimal experimental design. Kiefer
	\cite{Kie74} subsequently developed the general equivalence theory in
	1974. These results later provided the principal design-theoretic
	framework for studying the \(S\)-matrix problem. The conjecture itself
	was formulated by Sloane and Harwit \cite{SH76} in 1976 and was
	developed further in \cite{HS79,Slo79}.
	
	\item In 1987, Cheng  \cite{Che87} applied the Kiefer--Wolfowitz equivalence theory to
	prove the conjectured bound when \(n\) is odd. For even
	\(n\), he obtained the weaker estimate
	\[
	\norm{A^{-1}}_{F}
	>
	\frac{2\sqrt{n^2-2n+2}}{n}.
	\]

	\item In 2010, Zou, Jiang, and Hu \cite{ZJH10} established the conjectured bound for a
	spectrally restricted class of positive definite matrices.
	
	\item In 2012, Zou  \cite{Zou12} proved the conjectured estimate under an additional
	Frobenius-norm hypothesis.
	In the same year, Hu  \cite{Hu12} established the conjecture for positive-definite
	matrices.
	
	\item In 2013, Drnov\v{s}ek  \cite{Drn13} gave a short proof of Cheng's odd-dimensional
	result and even-dimensional bound using the Cauchy--Schwarz inequality
	and elementary optimization.
	
	\item In 2025, Frankel and Urschel 	\cite{FU25} proved the 
	conjecture for every \(n\geq 1000\). Their proof is based on a quantitative near-integrality argument.
	
\end{itemize}

\eqref{eq:mainbound} is equivalent to 
\begin{equation} 
	\tr\bigl((A^\top A)^{-1}\bigr) \geq \frac{4n^2}{(n+1)^2}. \label{eq:mainA} 
	\end{equation} 
Combining \eqref{eq:Acriterion} and \eqref{eq:mainA}, we obtain
\[
\E\norm{\widehat{\theta}-\theta}_2^2
\ge
\frac{4n^2}{(n+1)^2}\sigma^2,
\]
and hence
\[
\frac1n\sum_{j=1}^n
\operatorname{Var}(\widehat{\theta}_j)
\ge
\frac{4n}{(n+1)^2}\sigma^2.
\]
Whenever an \(S\)-matrix of order \(n\) exists, it is therefore an
\(A\)-optimal saturated design among all matrices with entries in
\([0,1]\). In that case, the binary design problem and its continuous
relaxation have the same optimal value. If no \(S\)-matrix of order
\(n\) exists, Theorem~\ref{thm:main} shows that the lower bound
\eqref{eq:mainbound} is not attained; it does not determine the exact
optimal value or compare the best binary design with the best
fractional design in that dimension.

At the level of the two universal lower bounds proved above, the
signed problem satisfies
\[
\norm{A^{-1}}_F^2\ge1
\]
by \eqref{eq:signedbound}, whereas the nonnegative problem satisfies
\[
\norm{A^{-1}}_F^2
\ge
\frac{4n^2}{(n+1)^2},
\]
whose right-hand side converges to \(4\) as \(n\to\infty\). This is a
comparison of the two lower-bound expressions, not necessarily of two
simultaneously attainable optima in the same dimension.

	Theorem~\ref{thm:main} also admits the following homogeneous formulation. 
	
	\begin{cor}
	Let  \(A=(a_{ij})\in\R^{n\times n}\) be nonsingular and 	 nonnegative. Let $\norm{A}_{\max}=\max_{i,j}a_{ij}$. Then
		\begin{equation*} 
			\norm{A^{-1}}_{F} \geq \frac{2n}{n+1}\norm{A}_{\max}^{-1}.
		\end{equation*} 
		Equality holds if and only if  \(A\) is a positive scalar multiple of an \(S\)-matrix.
	\end{cor}

\noindent\textbf{Sketch of the proof.}
The proof has three components.  First, in odd dimensions a fixed
matrix \(Y\) supplies a variational lower inequality for
\(\tr\bigl((AA^{\mathsf T})^{-1}\bigr)\).  The inequality is sharp exactly on the
middle Hamming layer of the cube, and its equality conditions recover
Hadamard orthogonality.  Second, for even \(n\) we prove the centered
estimate
\[
 \|D^\dagger\|_F^2\ge\frac{n-1}{n+1}
\]
for every rank-\((n-1)\) matrix \(D\) whose rows sum to zero and have
range at most \(2\).  The case \(n=4\) is polyhedral.  For \(n\ge6\), a
frame-potential functional has no negative minimum: rowwise second
variation forces any hypothetical minimizer onto projected vertices of
the cube, where parity supplies the missing gap.  Finally, an exact formula for \(A^{-1}\) separates its centered part from the one-dimensional direction lost under row centering.

\vspace{0.1in}
\noindent\textbf{Organization of the paper.}
In Section~\ref{s:2}, we give a simple proof of
Theorem~\ref{thm:main} in odd dimensions by combining a variational trace
inequality with a sharp maximization problem over the unit cube, and we
determine the corresponding equality cases. In Section~\ref{s:3}, we establish
a range-constrained estimate for the Moore--Penrose inverse of a
centered matrix. The case $n=4$ is treated through a
3-dimensional polyhedral reduction, while the case of even
\(n\ge6\) is handled using an inequality for projected sign matrices, a
rowwise second-variation argument, and a spectral-variance estimate.
In Section~\ref{s:4}, we prove the even case in Theorem~\ref{thm:main}  by means of an exact
inverse decomposition and a centered energy identity. We also treat the
case \(n=2\) directly and complete the proof of Theorem~\ref{thm:main}.

\vspace{0.1in}
\noindent\textbf{Acknowledgments.} 
The author thanks Professor Minghua Lin for introducing this problem to him and for the valuable discussions.
 He is also deeply grateful to Professor Zongben Xu for his kind support and concern regarding both his academic work and personal life.
This work was supported by the China Scholarship Council, the Young Elite Scientists Sponsorship Program for PhD Students (China Association for Science and Technology), and the Fundamental Research Funds for the Central Universities at Xi'an Jiaotong University (Grant No.~xzy022024045).

\section{The odd case}\label{s:2}

In this section, we give a simple proof of the odd-dimensional case of
Theorem~\ref{thm:main}.

We begin with a variational inequality for the inverse trace. The following completion-of-squares inequality will be applied to \(G=AA^{\mathsf T}\) with a suitably chosen symmetric matrix \(Y\).

\begin{lem}\label{lem:trace-inequality}
If \(G\in\R^{n\times n}\) is positive definite and
\(Y\in\R^{n\times n}\) is symmetric, then
\begin{equation*}
 \tr G^{-1}\ge 2\tr Y-\tr(YGY).
\end{equation*}
Equality holds if and only if \(Y=G^{-1}\).
\end{lem}

\begin{proof}
Notice that
\[
\tr G^{-1}-2\tr Y+\tr(YGY)=\bigl\|G^{1/2}(Y-G^{-1})\bigr\|_F^2.
\]

\end{proof}

To make the trace inequality effective, we next identify the sharp maximum of the quadratic form induced by the matrix \(Q\) over the unit cube. In odd dimensions, the maximizers lie precisely on a single Hamming layer, which will later determine the equality cases.

Let \(I_n\) denote the \(n\times n\) identity matrix, let \(\one\in\mathbb{R}^n\) be the all-ones vector, and put $ J_n=\one\one^{\mathsf T}$.
\begin{lem}\label{lem:odd-cube}
	Suppose that \(n\) is odd, and define
$
	Q=I_n-1/(n+1)J_n.
$
	Then
	\begin{equation*}
		\max_{x\in[0,1]^n}\|Qx\|_2^2=\frac n4.
	\end{equation*}
Equality holds if and only if \(x\in\{0,1\}^n\) and
$
\one^{\mathsf T}x=\frac{n+1}{2}.
$
\end{lem}

\begin{proof}
	Since
$
	Q^2=I_n-\frac{n+2}{(n+1)^2}J_n,
$
	we have
	\begin{equation}\label{eq:q-form}
		\|Qx\|_2^2
		=\sum_{i=1}^n x_i^2
		-\frac{n+2}{(n+1)^2}
		\left(\sum_{i=1}^n x_i\right)^2.
	\end{equation}
	The matrix \(Q^2\) is positive definite, so the function in
	\eqref{eq:q-form} is strictly convex.  Consequently, any maximizer
	over the cube must be a cube vertex.  At a vertex with \(k\) nonzero entries, the value of \eqref{eq:q-form} is
	\[
	q(k)=k-\frac{n+2}{(n+1)^2}k^2.
	\]
Since \(n\) is odd, the unique maximizing integer is $k=(n+1)/2$. The equality can occur only when \(x\in\{0,1\}^n\) has number of nonzero entries \((n+1)/2\), or equivalently, $ \one^{\mathsf T}x=(n+1)/2.$ Moreover, \[ q\left(\frac{n+1}{2}\right)=\frac n4. \]
	Finally, strict convexity rules out equality at any nonvertex. Indeed,
	if a nonvertex were written as a nontrivial convex combination of two
	distinct points of the cube, its value would be strictly smaller than
	the larger of the two endpoint values.  
\end{proof}

We now combine Lemma~\ref{lem:trace-inequality} with
Lemma~\ref{lem:odd-cube} to prove Theorem~\ref{thm:main} in odd
dimensions. The equality conditions in these two lemmas also provide
a complete characterization of the extremizers.

\begin{prop}\label{prop:odd}
	The conclusion of Theorem~\ref{thm:main} holds whenever \(n\) is odd.
\end{prop}

\begin{proof}
	Let \(G=AA^{\mathsf T}\), and set
	$ Q=I_n-\frac1{n+1}J_n$ and $Y=\frac4{n+1}Q$.
	Since $\tr Q=\frac{n^2}{n+1}$, Lemma~\ref{lem:trace-inequality} gives
	\begin{equation} \label{eq:odd1}
		\|A^{-1}\|_F^2 =\tr G^{-1} \ge \frac{8n^2}{(n+1)^2} -\frac{16}{(n+1)^2}\|QA\|_F^2.
	\end{equation}
	Applying Lemma~\ref{lem:odd-cube} to each of the \(n\) columns of \(A\)
	yields
	\begin{equation}\label{eq:odd2}
		\|QA\|_F^2\le\frac{n^2}{4}.
	\end{equation}
	Combining \eqref{eq:odd1} and \eqref{eq:odd2}, we obtain
	\begin{equation}\label{eq:odd3}
		\|A^{-1}\|_F^2\ge\frac{4n^2}{(n+1)^2}.
	\end{equation}
	
	It remains to characterize the equality case in \eqref{eq:odd3}.
	Suppose that equality holds. Then equality must hold in both
	\eqref{eq:odd1} and \eqref{eq:odd2}. By
	Lemma~\ref{lem:trace-inequality}, equality in \eqref{eq:odd1} implies
	\[
	G^{-1}=\frac4{n+1}Q.
	\]
	Moreover, by Lemma~\ref{lem:odd-cube}, equality in every columnwise
	estimate used in \eqref{eq:odd2} forces \(A\) to be a \(0\)--\(1\)
	matrix, each of whose columns contains exactly \((n+1)/2\) nonzero
	entries. Since $ Q^{-1}=I_n+J_n,$ we obtain
	\[
	AA^{\mathsf T} =\frac{n+1}{4}(I_n+J_n).
	\]
	By the equivalent Gram characterization of \(S\)-matrices in \eqref{eq:Sgram}, this shows
	that \(A\) is an \(S\)-matrix.
	
	Conversely, every \(S\)-matrix attains equality by \eqref{eq:Svalue}.
\end{proof}

\section{A range-constrained pseudoinverse estimate} \label{s:3}

First, we recall the definition of the Moore--Penrose inverse of a matrix; see
\cite[p.~40]{BG03}.
\begin{defin}
	Let \(M\in\R^{n\times n}\).  The \emph{Moore--Penrose inverse} of \(M\),
	denoted by \(M^\dagger\), is the unique matrix in \(\R^{n\times n}\)
	satisfying
	\[
	MM^\dagger M=M,
	\qquad
	M^\dagger MM^\dagger=M^\dagger,
	\]
	and
	\[
	(MM^\dagger)^{\mathsf T}=MM^\dagger,
	\qquad
	(M^\dagger M)^{\mathsf T}=M^\dagger M.
	\]
\end{defin}

For \(x\in\R^n\), set
\[
 \osc(x)=\max_j x_j-\min_j x_j.
\]
The central result of this section is the following estimate.

\begin{thm}\label{thm:centered}
Let \(n\ge4\) be even.  Suppose that \(D\in\R^{n\times n}\) satisfies
$
 \rank D=n-1, D\one=0,
$
and $
 \osc(d_i)\le2$ for all $1\le i\le n$,
where \(d_i^{\mathsf T}\) denotes the \(i\)-th row of \(D\).  Then
\begin{equation*}
 \|D^\dagger\|_F^2\ge\frac{n-1}{n+1}.
\end{equation*}
\end{thm}

We first describe the common row polytope.  Let
$
 P=I_n-\frac1nJ_n
$
be the orthogonal projection onto \(\one^\perp\), and define
\begin{equation*}
 \cK_n=\{x\in\one^\perp:\osc(x)\le2\}.
\end{equation*}

\begin{lem}\label{lem:row-polytope}
For every \(n\ge2\),
\begin{equation}\label{eq:projected-cube}
 \cK_n=P[-1,1]^n.
\end{equation}
Moreover,
\begin{equation}\label{eq:row-norm}
 \|x\|_2^2\le n \qquad (x\in\cK_n).
\end{equation}
If \(n\) is even, equality in \eqref{eq:row-norm} holds exactly at the
balanced sign vectors, that is, at vectors in \(\{-1,1\}^n\cap
\one^\perp\).
\end{lem}

\begin{proof} 
\textbf{Step 1: $P[-1,1]^n\subset \cK_n$.}	Let \(y\in[-1,1]^n\). Since
	\[
	Py
	=
	y-\frac{1}{n}\bigl(\one^{\mathsf T}y\bigr)\one,
	\]
	the vector \(Py\) is obtained from \(y\) by subtracting the same
	constant from every coordinate. Hence, for all \(j,k\),
	\[
	(Py)_j-(Py)_k=y_j-y_k.
	\]
	It follows that
	\[
	\osc(Py)
	=
	\max_{j,k}\bigl((Py)_j-(Py)_k\bigr)
	=
	\max_{j,k}(y_j-y_k)
	=
	\osc(y).
	\]
	Since \(y\in[-1,1]^n\), we have
	\[
	\osc(y)\le 2.
	\]
	Moreover,
	\[
	\one^{\mathsf T}Py=0,
	\]
	because \(P\) is the orthogonal projection onto \(\one^\perp\).
	Therefore \(Py\in\one^\perp\) and \(\osc(Py)\le2\), so
	\[
	Py\in\cK_n.
	\]
	As \(y\in[-1,1]^n\) was arbitrary, this proves
	\[
	P[-1,1]^n\subseteq\cK_n.
	\]


\textbf{Step 2: $ \cK_n\subset P[-1,1]^n$.} Conversely, let \(x\in\cK_n\), and set
\[
a=\min_j x_j,
\qquad
b=\max_j x_j,
\qquad
c=\frac{a+b}{2}.
\]
Since \(x\in\cK_n\), we have
\[
b-a=\osc(x)\le2.
\]
For every \(j\),
\[
a\le x_j\le b,
\]
and therefore
\[
-\frac{b-a}{2}
=
a-c
\le
x_j-c
\le
b-c
=
\frac{b-a}{2}.
\]
Because \(b-a\le2\), it follows that
\[
-1\le x_j-c\le1
\qquad (1\le j\le n).
\]
Hence
\[
x-c\one\in[-1,1]^n.
\]
Moreover, \(x\in\one^\perp\), so \(Px=x\), while \(P\one=0\).
Consequently,
\[
P(x-c\one)
=
Px-cP\one
=
x.
\]
Thus \(x\) is the projection under \(P\) of a vector in
\([-1,1]^n\), and hence
\[
x\in P[-1,1]^n.
\]
This proves
\[
\cK_n\subseteq P[-1,1]^n.
\]
Together with the reverse inclusion established above, we obtain
\eqref{eq:projected-cube}.

\textbf{Step 3: norm estimate.}  Recall \(a=\min_j x_j\) and \(b=\max_j x_j\).  Since \(x\perp\one\), we
have \(a\le0\le b\).  The inequality
\((b-x_j)(x_j-a)\ge0\), summed over \(j\), gives
\[
 \|x\|_2^2\le-nab\le n\frac{(b-a)^2}{4}\le n.
\]
Equality requires \(b-a=2\), \(b=-a=1\), and every coordinate to equal
one of the endpoints.  The zero-sum condition then forces equal numbers
of \(1\)'s and \(-1\)'s.  Conversely, every balanced sign vector plainly
gives equality.
\end{proof}

The proof of Theorem~\ref{thm:centered} has two parts. The case $n=4$
has a small but exceptional coefficient polytope, which we handle
first.  The case $n\ge 6$ admits a uniform frame-potential
argument.

\subsection{The case $n=4$ in Theorem~\ref{thm:centered}}

Throughout this subsection, let
\begin{equation}\label{eq:uvw}
	u=(1,1,-1,-1)^{\mathsf T},\qquad
	v=(1,-1,1,-1)^{\mathsf T},\qquad
	w=(1,-1,-1,1)^{\mathsf T},
\end{equation}
and define
\begin{equation}\label{eq:K4}
	\mathcal{P}_4
	=
	\left\{
	(a,b,c)\in\mathbb{R}^3:
	|a|+|b|\le1,\ 
	|a|+|c|\le1,\ 
	|b|+|c|\le1
	\right\}.
\end{equation}

We first reduce the problem to a 3-dimensional Gram matrix.

\begin{lem}\label{lem:n4-reduction}
	Let \(D\in\mathbb{R}^{4\times4}\) satisfy
	\[
	D\one=0,
	\qquad
	\rank D=3,
	\qquad
	\osc(d_i)\le2
	\quad (1\le i\le4),
	\]
where \(d_i^{\mathsf T}\) is the \(i\)-th row of \(D\). Write
$
d_i=a_i u+b_i v+c_i w,
$
where \(u,v,w\) are defined by \eqref{eq:uvw}. Let
$
r_i=(a_i,b_i,c_i)^{\mathsf T}\in\mathbb{R}^3,
$
let \(M\in\mathbb{R}^{4\times3}\) have \(i\)-th row
\(r_i^{\mathsf T}\), and put
$
\Gamma=M^{\mathsf T}M.
$
	Then \(r_i\in\mathcal{P}_4\) in \eqref{eq:K4} for every \(i\), the matrix
	\(\Gamma\) is positive definite, and
	\begin{equation}\label{eq:n4-pseudo}
		\|D^\dagger\|_F^2
		=
		\frac14\tr\Gamma^{-1}.
	\end{equation}
\end{lem}

\begin{proof}
	Since \(D\one=0\), every row \(d_i\) belongs to
	\(\one^\perp\). The vectors \(u,v,w\) form an orthogonal basis of
	\(\one^\perp\), and each has squared norm \(4\). Hence each row has a
	unique representation
	\[
	d_i=a_i u+b_i v+c_i w.
	\]
	
	For \(d=au+bv+cw\), its four coordinates are
	\[
	a+b+c,\qquad
	a-b-c,\qquad
	-a+b-c,\qquad
	-a-b+c.
	\]
	The six quantities obtained by halving the pairwise coordinate
	differences are
	\[
	a\pm b,\qquad
	a\pm c,\qquad
	b\pm c.
	\]
	Therefore, \(\osc(d)\le2\) if and only if all six of these quantities
	have modulus at most \(1\). Since
	\[
	\max\{|a+b|,|a-b|\}=|a|+|b|,
	\]
	and similarly for the other two pairs, this condition is equivalent to
	\[
	|a|+|b|\le1,\qquad
	|a|+|c|\le1,\qquad
	|b|+|c|\le1.
	\]
	Thus \(r_i\in\mathcal{P}_4\).
	
	Let \(U\in\mathbb{R}^{3\times4}\) be the matrix whose rows are
	\(u^{\mathsf T},v^{\mathsf T},w^{\mathsf T}\). Then
	\[
	D=MU,
	\qquad
	UU^{\mathsf T}=4I_3,
	\]
	and consequently
	\[
	DD^{\mathsf T}
	=
	MUU^{\mathsf T}M^{\mathsf T}
	=
	4MM^{\mathsf T}.
	\]
	The nonzero eigenvalues of \(D^{\mathsf T}D\) coincide with those of
	\(DD^{\mathsf T}\), while the nonzero eigenvalues of
	\(MM^{\mathsf T}\) coincide with those of
	\(\Gamma=M^{\mathsf T}M\). Hence the three nonzero eigenvalues of
	\(D^{\mathsf T}D\) are four times the eigenvalues of \(\Gamma\).
	
	Since \(\rank D=3\), we also have \(\rank M=3\), so \(\Gamma\) is
	positive definite. Summing the reciprocals of the three nonzero
	eigenvalues of \(D^{\mathsf T}D\) gives
	\[
	\|D^\dagger\|_F^2
	=
	\frac14\tr\Gamma^{-1}.
	\]
\end{proof}

The next lemma describes the row polytope appearing in
Lemma~\ref{lem:n4-reduction}.

\begin{lem}\label{lem:n4-polytope}
The polytope \(\mathcal{P}_4\) in \eqref{eq:K4} is the convex hull of
the six axial points
\[
\pm e_1=\pm(1,0,0),\qquad
\pm e_2=\pm(0,1,0),\qquad
\pm e_3=\pm(0,0,1),
\]
and the eight half-cube points
\[
\left(\pm\frac12,\pm\frac12,\pm\frac12\right).
\]
	In particular,
	\[
	\|r\|_2\le1
	\qquad
	(r\in\mathcal{P}_4).
	\]
\end{lem}

\begin{proof}
	All the listed points belong to \(\mathcal{P}_4\). Since
	\(\mathcal{P}_4\) is convex, their convex hull is contained in
	\(\mathcal{P}_4\).
	
	For the reverse inclusion, let \((a,b,c)\in\mathcal{P}_4\). The set
	\(\mathcal{P}_4\) and the collection of listed points are invariant
	under coordinate permutations and independent sign changes. We may
	therefore assume that
	\[
	a\ge b\ge c\ge0.
	\]
	Then
	\begin{align*}
		(a,b,c)
		={}&(a-b)e_1
		+(b+c)\left(\frac12,\frac12,\frac12\right)\\
		&+(b-c)\left(\frac12,\frac12,-\frac12\right)
		+(1-a-b)0.
	\end{align*}
	The coefficients are nonnegative because
	\[
	a\ge b\ge c\ge0
	\qquad\text{and}\qquad
	a+b\le1,
	\]
	and their sum is equal to \(1\). Moreover, \(0\) is the midpoint of
	any pair of opposite axial points. Hence \((a,b,c)\) belongs to the
	convex hull of the listed points.
	
	Every axial point has norm \(1\), and every half-cube point has norm
	\(\sqrt{3}/2\). The convexity of the Euclidean norm therefore gives
	\[
	\|r\|_2\le1
	\qquad
	(r\in\mathcal{P}_4).
	\]
\end{proof}

We next establish the determinant estimate needed in the spectral
argument.

\begin{lem}\label{lem:n4-determinant}
	Let \(M\in\mathbb{R}^{4\times3}\) have \(i\)-th row
	\(r_i^{\mathsf T}\), where
	\(r_1,\ldots,r_4\in\mathcal{P}_4\), and put
$
	\Gamma=M^{\mathsf T}M.
$
	Then
	\begin{equation}\label{eq:det-Gamma}
		\det\Gamma\le2.
	\end{equation}
\end{lem}

\begin{proof}
	We first show that each row can be replaced by one of the generating
	points in Lemma~\ref{lem:n4-polytope} without decreasing the
	determinant.
	
	Fix \(r_2,r_3,r_4\), and set
	\[
	\Gamma_0
	=
	r_2r_2^{\mathsf T}
	+r_3r_3^{\mathsf T}
	+r_4r_4^{\mathsf T}.
	\]
	Regarding the first row as variable, define
	\[
	f(r)=\det(\Gamma_0+rr^{\mathsf T}).
	\]
	Clearly, $f(r_1)=\det\Gamma$.
	The rank-one determinant identity in adjugate form gives
	\[
	f(r)
	=
	\det\Gamma_0
	+r^{\mathsf T}\adj(\Gamma_0)r.
	\]
	
	Since \(\Gamma_0\) is positive semidefinite, write
	\[
	\Gamma_0
	=
	O\diag(\lambda_1,\lambda_2,\lambda_3)O^{\mathsf T},
	\qquad
	\lambda_i\ge0,
	\]
	where \(O\) is orthogonal. Then
	\[
	\adj(\Gamma_0)
	=
	O\diag
	\bigl(
	\lambda_2\lambda_3,
	\lambda_1\lambda_3,
	\lambda_1\lambda_2
	\bigr)O^{\mathsf T},
	\]
	so \(\adj(\Gamma_0)\) is also positive semidefinite. It follows that
	\(f\) is a convex function of \(r\).
	
	Let \(\mathcal{V}\) denote the set of generating points listed in
	Lemma~\ref{lem:n4-polytope}. Since
	\(\mathcal{P}_4=\operatorname{conv}(\mathcal{V})\), we may write
	\[
	r_1=\sum_{k=1}^N\alpha_k v_k,
	\qquad
	v_k\in\mathcal{V},
	\qquad
	\alpha_k\ge0,
	\qquad
	\sum_{k=1}^N\alpha_k=1.
	\]
	By convexity,
	\[
	f(r_1)
	\le
	\sum_{k=1}^N\alpha_k f(v_k)
	\le
	\max_{1\le k\le N}f(v_k).
	\]
	Thus \(r_1\) can be replaced by some \(v_k\in\mathcal{V}\) without
	decreasing \(\det\Gamma\). Repeating this argument successively for
	the remaining three rows shows that it suffices to prove the bound
	when all four rows belong to \(\mathcal{V}\).
	
	Suppose that \(\ell\) of these rows are axial and the remaining
	\(4-\ell\) are half-cube points. Let \(\ell_j\) be the number of axial
	rows lying on the \(j\)-th coordinate axis. Then
	\[
	\ell_1+\ell_2+\ell_3=\ell,
	\]
	and the \(j\)-th diagonal entry of \(\Gamma\) is
	\[
	\Gamma_{jj}
	=
	\ell_j+\frac{4-\ell}{4}.
	\]
	Indeed, an axial row on the \(j\)-th axis contributes \(1\) to
	\(\Gamma_{jj}\), whereas every half-cube row contributes \(1/4\) to
	each diagonal entry.
	
Hadamard's determinant inequality
\cite[p.~505, Theorem~7.8.1]{HJ13} gives
\[
\det\Gamma
\le
\prod_{j=1}^3\Gamma_{jj}
=
\prod_{j=1}^3
\left(
\ell_j+\frac{4-\ell}{4}
\right).
\]
	For fixed \(\ell\), this product is maximized when the integers
	\(\ell_1,\ell_2,\ell_3\) are as evenly distributed as possible.
	Indeed, if \(\ell_j\ge\ell_k+2\), then transferring one unit from
	\(\ell_j\) to \(\ell_k\) changes the corresponding two factors by
	\[
	\begin{aligned}
		&\quad\left(
		\ell_j-1+\frac{4-\ell}{4}
		\right)
		\left(
		\ell_k+1+\frac{4-\ell}{4}
		\right)-
		\left(
		\ell_j+\frac{4-\ell}{4}
		\right)
		\left(
		\ell_k+\frac{4-\ell}{4}
		\right)\\
		&=
		\ell_j-\ell_k-1>0.
	\end{aligned}
	\]
	Thus, up to permutation, the maximizing distributions and the
	resulting bounds are
	\[
	\begin{array}{c|c|c}
		\ell
		&
		(\ell_1,\ell_2,\ell_3)
		&
		\displaystyle
		\prod_{j=1}^3
		\left(
		\ell_j+\frac{4-\ell}{4}
		\right)
		\\ \hline
		0&(0,0,0)&1\\[2mm]
		1&(1,0,0)&\dfrac{63}{64}\\[2mm]
		2&(1,1,0)&\dfrac98\\[2mm]
		3&(1,1,1)&\dfrac{125}{64}\\[2mm]
		4&(2,1,1)&2
	\end{array}
	\]
	The largest of these values is \(2\), which proves
	\eqref{eq:det-Gamma}.
\end{proof}

The following lemma handles the case in which the trace of
\(\Gamma\) is close to its maximal possible value.

\begin{lem}\label{lem:n4-small-eigenvalue}
	Let \(M\in\mathbb{R}^{4\times3}\) have \(i\)-th row
	\(r_i^{\mathsf T}\), where
	\(r_1,\ldots,r_4\in\mathcal{P}_4\), and put
$
	\Gamma=M^{\mathsf T}M.
$
	If
$
	\tr\Gamma>\frac{15}{4},
$
	then
	\begin{equation}\label{eq:lambda-min-four}
		\lambda_{\min}(\Gamma)<\frac{37}{36}.
	\end{equation}
\end{lem}

\begin{proof}
	By Lemma~\ref{lem:n4-polytope},
$	\|r_i\|_2^2\le1.
$
	Set
$
	\delta_i=1-\|r_i\|_2^2.
$
	Then \(\delta_i\ge0\), and
	\[
	\sum_{i=1}^4\delta_i
	=
	4-\tr\Gamma
	<
	\frac14.
	\]
	In particular,
	\[
	\|r_i\|_2^2=1-\delta_i\ge 1-	\sum_{i=1}^4\delta_i>\frac34
	\qquad
	(1\le i\le4).
	\]
	
	Let \(p_i\) be the largest absolute coordinate of \(r_i\). If
	\(p_i\le1/2\), then
	\[
	\|r_i\|_2^2
	\le
	3p_i^2
	\le
	\frac34,
	\]
	which is impossible. Suppose instead that
$
	1/2< p_i\le5/6.
$
	Since \(r_i\in\mathcal{P}_4\), each of the other two absolute
	coordinates is at most \(1-p_i\). Hence
	\[
	\|r_i\|_2^2
	\le
	p_i^2+2(1-p_i)^2
	\le
	\frac34,
	\]
	where the last inequality follows from
	\[
	p_i^2+2(1-p_i)^2-\frac34
	=
	3\left(p_i-\frac12\right)
	\left(p_i-\frac56\right)
	\le0.
	\]
	This is again impossible. Therefore
	\[
	p_i>\frac56.
	\]
	
	The coordinate attaining \(p_i\) is unique. Indeed, if two coordinates
	had absolute value \(p_i\), their absolute values would sum to more
	than \(1\), contradicting the definition of \(\mathcal{P}_4\). We
	call this unique coordinate the principal coordinate of \(r_i\).
	
	Put $
	h_i=1-p_i.$
	The two nonprincipal coordinates of \(r_i\) have absolute value at
	most \(h_i\). Consequently,
		\[
\delta_i
\ge
1-p_i^2-2h_i^2
=
h_i(3p_i-1)
\ge
\frac32h_i.
\]
Consequently,
\[
\sum_{i=1}^4 h_i
\le
\frac23\sum_{i=1}^4\delta_i
<
\frac16.
\]
	It follows that
	\[
	\sum_{i=1}^4 h_i
	<
	\frac23\sum_{i=1}^4\delta_i
	<
	\frac16.
	\]
	
	Assign each row to its principal coordinate. If some coordinate,
	say the \(j_0\)-th, is not assigned to any row, then
	\[
	\Gamma_{j_0j_0}
	=
	\sum_{i=1}^4(r_i)_{j_0}^2
	\le
	\sum_{i=1}^4h_i^2
	\le
	\left(\sum_{i=1}^4h_i\right)^2
	<
	\frac1{36}.
	\]
	Otherwise, all three coordinates are assigned. Since four rows are
	assigned to three coordinates, the assignment counts must be
	\((2,1,1)\), up to permutation. Choose a coordinate assigned to
	exactly one row. The corresponding diagonal entry satisfies
	\[
	\Gamma_{j_0j_0}
	\le
	1+\sum_{i=1}^4h_i^2
	\le
	1+\left(\sum_{i=1}^4h_i\right)^2
	<
	\frac{37}{36}.
	\]
	In either case, the Rayleigh--Ritz principle  \cite[p.~235, Theorem~4.2.2(c)]{HJ13} yields
	\[
	\lambda_{\min}(\Gamma)
	\le
	e_{j_0}^{\mathsf T}\Gamma e_{j_0}
	=
	\Gamma_{j_0j_0}
	<
	\frac{37}{36},
	\]
	which proves \eqref{eq:lambda-min-four}.
\end{proof}

We can now prove the case $n=4$ in Theorem~\ref{thm:centered}.

\begin{prop}\label{prop:n-four}
	The conclusion of Theorem~\ref{thm:centered} holds for \(n=4\).
\end{prop}

\begin{proof}
	Let \(D\) satisfy the hypotheses of
	Theorem~\ref{thm:centered} with \(n=4\), and let
	\(\Gamma\) be the positive-definite matrix supplied by
	Lemma~\ref{lem:n4-reduction}. By \eqref{eq:n4-pseudo},
	it is enough to prove that
	\[
	\tr\Gamma^{-1}\ge\frac{12}{5}.
	\]
	
	Suppose first that
	$\tr\Gamma\le 15/4.$
	The arithmetic--harmonic mean inequality applied to the three
	eigenvalues of \(\Gamma\) gives
	\[
	\tr\Gamma^{-1}
	\ge
	\frac9{\tr\Gamma}
	\ge
	\frac{12}{5}.
	\]
	
	It remains to consider
	$\tr\Gamma>15/4.$
	Let
$
	0<x\le y\le z
$
	be the eigenvalues of \(\Gamma\). By
	Lemma~\ref{lem:n4-small-eigenvalue},
	\[
	x<\frac{37}{36},
	\]
	while Lemma~\ref{lem:n4-determinant} gives
	\[
	xyz=\det\Gamma\le2.
	\]
	Hence
	\[
	yz\le\frac2x.
	\]
	It follows that
	\begin{align*}
		\tr\Gamma^{-1}
		&=
		\frac1x+\frac1y+\frac1z\\
		&\ge
		\frac1x+\frac2{\sqrt{yz}}\\
		&\ge
		\frac1x+\sqrt{2x}\\
		&>
		\frac{36}{37}+\sqrt{\frac{37}{18}}\\
		&>
		\frac{12}{5}.
	\end{align*}
	Here the third inequality is due to the function
$
	s\longmapsto s^{-1}+\sqrt{2s}
$
	is decreasing on \((0,37/36]\), because
$
37/36<2^{1/3}.
$
	The last displayed inequality may be verified directly.
	
	Thus, in both cases,
	\[
	\tr\Gamma^{-1}\ge\frac{12}{5}.
	\]
	Finally, \eqref{eq:n4-pseudo} gives
	\[
	\|D^\dagger\|_F^2
	=
	\frac14\tr\Gamma^{-1}
	\ge
	\frac35
	=
	\frac{n-1}{n+1},
	\]
	as required.
\end{proof}

\subsection{The case \(n\ge6\) in Theorem~\ref{thm:centered}}

We now prove the case \(n\ge6\) in Theorem~\ref{thm:centered}. The proof consists of three steps. We first establish a
quantitative estimate for projected sign matrices. We then extend this
estimate to the entire row polytope by a second-variation argument.
Finally, we convert the resulting spectral-variance estimate into the
required bound for the Moore--Penrose inverse.

\begin{lem}\label{lem:projected-sign}
	Let \(n\ge6\) be even. Let
	$P=I_n-\frac1nJ_n$
	and \(R\in\{-1,1\}^{n\times n}\). Define
	\[
	D=RP,
	\qquad
	G=D^{\mathsf T}D,
	\qquad
	T=\tr G,\qquad
	V
	=
	\tr(G^2)-\frac{T^2}{n-1}.
	\]
	Then
	\begin{equation}\label{eq:projected-sign-gap}
		V
		\ge
		\frac{n^2(n-2)}{n-1}
		\bigl(T-(n^2-1)\bigr).
	\end{equation}
\end{lem}

\begin{proof}
	Set
	\[
	\rho=R\one,
	\qquad
	\delta=\frac{\|\rho\|_2^2}{n},
	\qquad
	N=R^{\mathsf T}R.
	\]
Then
	\[
	G
	=
	D^{\mathsf T}D
	=
	PR^{\mathsf T}RP
	=
	PNP.
	\]Since \(G\ge0\) and \(G\one=0\), the Cauchy--Schwarz
	inequality applied to the eigenvalues of
	\(G|_{\one^\perp}\) gives
	\[
	T^2\le (n-1)\tr(G^2).
	\]
	Consequently,
	\[
	V=\tr(G^2)-\frac{T^2}{n-1}\ge0.
	\]
	
We first compute \(T=\tr G\):
	\begin{equation*}
		T=
		\tr(PNP)\notag
		=
		\tr(NP)\notag
		=
		\tr N-\frac1n\one^{\mathsf T}N\one.
	\end{equation*}
Since
$
	\tr N
	=
	\|R\|_F^2
	=
	n^2
$ and
$
	\one^{\mathsf T}N\one
	=
	\|R\one\|_2^2
	=
	\|\rho\|_2^2
	=
	n\delta,
$
	\begin{equation}\label{eq:T-delta}
		T=n^2-\delta.
	\end{equation}
	
	If \(\delta\ge1\), then
	\[
	T-(n^2-1)=1-\delta\le0.
	\]
	Since \(V\ge0\), inequality
	\eqref{eq:projected-sign-gap} follows immediately. 
	
Assume now that \(0\le\delta<1\). We next obtain a lower bound for \(\tr(N^2)\). Let \(R^{(j)}\) denote
	the \(j\)-th column of \(R\). Since \(N=R^{\mathsf T}R\),
	\[
	N_{jk}
	=
	\left\langle R^{(j)},R^{(k)}\right\rangle.
	\]
	For \(j=k\), every entry of \(R^{(j)}\) has modulus \(1\), and hence
	\[
	N_{jj}=n.
	\]
	Now let \(j\ne k\). Suppose that the \(j\)-th and \(k\)-th columns of
	\(R\) agree in \(a\) positions and disagree in \(b\) positions. Then
$
	a+b=n
$
	and
	\[
	N_{jk}=a-b=n-2b.
	\]
	Because \(n\) is even, \(N_{jk}\) is an even integer. We claim that
	\[
	N_{jk}^2\ge-2N_{jk}.
	\]
	Indeed, this is immediate when \(N_{jk}\ge0\). If \(N_{jk}<0\), then
	the evenness of \(N_{jk}\) implies \(N_{jk}\le-2\), and therefore
	\[
	N_{jk}^2
	=
	|N_{jk}|^2
	\ge
	2|N_{jk}|
	=
	-2N_{jk}.
	\]
	Consequently,
	\begin{equation}\label{eq:N-square-first}
		\tr(N^2)
		=
		\sum_{j,k=1}^nN_{jk}^2
		=
		n^3+\sum_{j\ne k}N_{jk}^2\ge
		n^3-2\sum_{j\ne k}N_{jk}.
	\end{equation}
Moreover, \begin{equation*} \sum_{j\ne k}N_{jk} = \sum_{j,k=1}^n N_{jk} - \sum_{j=1}^n N_{jj} = \one^{\mathsf T}N\one-\tr N= n\delta-n^2. \end{equation*}
	Substitution into \eqref{eq:N-square-first} gives
	\begin{equation}\label{eq:N-square}
		\tr(N^2)
		\ge
		n^3+2n^2-2n\delta.
	\end{equation}
	We also need an upper bound for \(\|N\one\|_2\). Since
	\(\rho=R\one\), every entry of \(\rho\) is a sum of \(n\) numbers in
	\(\{-1,1\}\). Because \(n\) is even, every entry of \(\rho\) is an
	even integer.
	Let
$
	m=\bigl|\{i:\rho_i\ne0\}\bigr|.
$
	Every nonzero entry of \(\rho\) has modulus at least \(2\), so
$
	4m
	\le
	\|\rho\|_2^2
	=
	n\delta.
$
	Thus
	\begin{equation}\label{eq:support-rho}
		m\le\frac{n\delta}{4}.
	\end{equation}
	For each \(j\),
	\[
	(N\one)_j
	=
	(R^{\mathsf T}\rho)_j
	=
	\sum_{i:\rho_i\ne0}R_{ij}\rho_i.
	\]
	The Cauchy--Schwarz inequality and \eqref{eq:support-rho} give
	\begin{equation*}
		\bigl((N\one)_j\bigr)^2
		\le
		m\sum_{i:\rho_i\ne0}\rho_i^2
		=
		m\|\rho\|_2^2\le
		\frac{n^2\delta^2}{4}.
	\end{equation*}
	Summing over \(j\), we obtain
	\begin{equation}\label{eq:None-bound}
		\|N\one\|_2^2
		\le
		\frac{n^3\delta^2}{4}.
	\end{equation}
	We now pass from \(N\) to \(G=PNP\). Expanding
	\(P=I_n-J_n/n\) and using cyclicity of the trace yields
	\begin{align}
		\tr(G^2)
		&=
		\tr(NPNP)\notag\\
		&=
		\tr(N^2)
		-\frac2n\one^{\mathsf T}N^2\one
		+\frac1{n^2}
		\bigl(\one^{\mathsf T}N\one\bigr)^2\notag\\
		&=
		\tr(N^2)
		-\frac2n\|N\one\|_2^2
		+\delta^2.
		\label{eq:PS-square}
	\end{align}
	Combining \eqref{eq:N-square}, \eqref{eq:None-bound}, and
	\eqref{eq:PS-square}, we obtain
	\begin{equation}\label{eq:G-square-sign}
		\tr(G^2)
		\ge
		n^3+2n^2-2n\delta
		-\left(\frac{n^2}{2}-1\right)\delta^2.
	\end{equation}
	Finally, by the definitions of \(T\) and \(V\),
	\[
	(n-1)V
	=
	(n-1)\tr(G^2)-T^2.
	\]
	Using \eqref{eq:T-delta} and \eqref{eq:G-square-sign}, a direct
	simplification gives
	\begin{align}
		&(n-1)
		\left[
		V-
		\frac{n^2(n-2)}{n-1}
		\bigl(T-(n^2-1)\bigr)
		\right]\notag\\
		&\quad\ge
		\delta
		\left[
		n(n^2-2n+2)
		-
		\left(
		\frac{n^2(n-1)}2-n+2
		\right)\delta
		\right].
		\label{eq:projected-final}
	\end{align}
	The coefficient
	\[
	\frac{n^2(n-1)}2-n+2
	\]
	is positive. Since \(0\le\delta<1\), the expression in square
	brackets in \eqref{eq:projected-final} is bounded below by
	\begin{equation*}
		n(n^2-2n+2)
		-
		\left(
		\frac{n^2(n-1)}2-n+2
		\right)=
		\frac{(n-1)(n^2-2n+4)}2
		>0.
	\end{equation*}
	The right-hand side of \eqref{eq:projected-final} is therefore
	nonnegative, which proves \eqref{eq:projected-sign-gap}.
\end{proof}

We isolate the second-variation computation used in the extension from
projected sign matrices to arbitrary admissible matrices.

\begin{lem}\label{lem:row-second-variation}
	Let \(n\ge6\), and let \(D\in\mathbb{R}^{n\times n}\) satisfy
	\(D\one=0\). Define
	\[
	G=D^{\mathsf T}D,
	\qquad
	T=\tr G,
	\]
	and
	\[
	\Phi(D)
	=
	\tr(G^2)-\frac{T^2}{n-1}
	-
	\frac{n^2(n-2)}{n-1}
	\bigl(T-(n^2-1)\bigr).
	\]
	Fix \(x,z\in\one^\perp\), and suppose that \(x^{\mathsf T}\) is
	one of the rows of \(D\). For \(\tau\in\mathbb{R}\), let
	\(D_\tau\) be obtained from \(D\) by replacing the row
	\(x^{\mathsf T}\) with \((x+\tau z)^{\mathsf T}\). Then
	\begin{equation}\label{eq:row-second-variation}
		\begin{aligned}
			\frac14
			\left.
			\frac{d^2}{d\tau^2}
			\right|_{\tau=0}
			\Phi(D_\tau)
			={}&
			z^{\mathsf T}Gz
			+
			\left(1-\frac{2}{n-1}\right)
			(x^{\mathsf T}z)^2\\
			&+
			\left(
			\|x\|_2^2
			-\frac{T}{n-1}
			-\frac{n^2(n-2)}{2(n-1)}
			\right)
			\|z\|_2^2.
		\end{aligned}
	\end{equation}
\end{lem}

\begin{proof}
	Only one row of \(D\) is changed. Therefore
	\[
	D_\tau^{\mathsf T}D_\tau
	=
	G
	+\tau(xz^{\mathsf T}+zx^{\mathsf T})
	+\tau^2zz^{\mathsf T}.
	\]
	If
$
	T_\tau=\tr(D_\tau^{\mathsf T}D_\tau),
$
	then
	\[
	T_\tau
	=
	T+2\tau x^{\mathsf T}z+\tau^2\|z\|_2^2.
	\]
	
	Direct differentiation gives
	\[
	\left.
	\frac{d^2}{d\tau^2}
	\right|_{\tau=0}
	\tr\bigl((D_\tau^{\mathsf T}D_\tau)^2\bigr)
	=
	4z^{\mathsf T}Gz
	+4\|x\|_2^2\|z\|_2^2
	+4(x^{\mathsf T}z)^2
	\]
	and
	\[
	\left.
	\frac{d^2}{d\tau^2}
	\right|_{\tau=0}
	T_\tau^2
	=
	8(x^{\mathsf T}z)^2
	+4T\|z\|_2^2.
	\]
	Also,
	\[
	\left.
	\frac{d^2}{d\tau^2}
	\right|_{\tau=0}
	T_\tau
	=
	2\|z\|_2^2.
	\]
	Substituting these three identities into the definition of
	\(\Phi(D_\tau)\) gives \eqref{eq:row-second-variation}.
\end{proof}

We now extend Lemma~\ref{lem:projected-sign} to every matrix whose rows
belong to the row polytope.

\begin{lem}\label{lem:frame-gap}
	Let \(n\ge6\) be even, and define
	\[
	\mathcal{K}_n
	=
	\left\{
	x\in\mathbb{R}^n:
	\one^{\mathsf T}x=0
	\ \text{and}\
	\osc(x)\le2
	\right\}.
	\]
Let \(D\in\mathbb{R}^{n\times n}\) have rows
\(d_1^{\mathsf T},\ldots,d_n^{\mathsf T}\), where
\(d_1,\ldots,d_n\in\mathcal{K}_n\). Define
	\[
	G=D^{\mathsf T}D,
	\qquad
	T=\tr G,
	\qquad
	V
	=
	\tr(G^2)-\frac{T^2}{n-1}.
	\]
	Then
	\begin{equation}\label{eq:frame-gap}
		V
		\ge
		\frac{n^2(n-2)}{n-1}
		\bigl(T-(n^2-1)\bigr).
	\end{equation}
\end{lem}

\begin{proof}
	For every matrix \(X\in\mathbb{R}^{n\times n}\) whose rows belong to
	\(\mathcal{K}_n\), define
	\[
	G(X)=X^{\mathsf T}X,
	\qquad
	T(X)=\tr G(X),
	\]
	and
	\[
	V(X)
	=
	\tr\bigl(G(X)^2\bigr)
	-
	\frac{T(X)^2}{n-1}.
	\]
	We also define
	\[
	\Phi(X)
	=
	V(X)
	-
	\frac{n^2(n-2)}{n-1}
	\bigl(T(X)-(n^2-1)\bigr).
	\]
	For the matrix \(D\) in the statement of the lemma,
	\eqref{eq:frame-gap} is precisely the assertion that
	\[
	\Phi(D)\ge0.
	\]
	
	Suppose, for a contradiction, that \(\Phi\) takes a negative value on
	the set of matrices whose rows belong to \(\mathcal{K}_n\). By
	Lemma~\ref{lem:row-polytope}, the set \(\mathcal{K}_n\) is compact.
	Hence the set of such matrices is compact, being the Cartesian product
	of \(n\) copies of \(\mathcal{K}_n\). Since \(\Phi\) is continuous,
	it attains its global minimum on this set.
	
	Let \(X_\ast\) be an admissible matrix at which this minimum is
	attained. By the contradiction assumption,
	\[
	\Phi(X_\ast)<0.
	\]
	For convenience, set
$
	G_\ast=X_\ast^{\mathsf T}X_\ast,
	T_\ast=\tr G_\ast,
$
	and
$
	V_\ast
	=
	\tr(G_\ast^2)
	-
	\frac{T_\ast^2}{n-1}.
$
Let \(x_1^{\mathsf T},\ldots,x_n^{\mathsf T}\) denote the rows of
\(X_\ast\).
Since every \(x_i\) belongs to \(\mathcal{K}_n\),
	Lemma~\ref{lem:row-polytope} gives
	\[
	\|x_i\|_2^2\le n
	\qquad (1\le i\le n).
	\]
	Consequently,
	\[
	T_\ast
	=
	\|X_\ast\|_F^2
	=
	\sum_{i=1}^n\|x_i\|_2^2
	\le
	n^2.
	\]
	
	Since \(X_\ast\one=0\), the matrix \(G_\ast\) annihilates \(\one\).
	If \(\lambda_1,\ldots,\lambda_{n-1}\) are the eigenvalues of
	\(G_\ast|_{\one^\perp}\), including zero eigenvalues when necessary,
	then
	\[
	T_\ast
	=
	\sum_{i=1}^{n-1}\lambda_i, \qquad
	V_\ast
	=
	\sum_{i=1}^{n-1}
	\left(
	\lambda_i-\frac{T_\ast}{n-1}
	\right)^2.
	\]
	In particular,
	\[
	V_\ast\ge0.
	\]
	
	Since
	\[
	\Phi(X_\ast)
	=
	V_\ast
	-
	\frac{n^2(n-2)}{n-1}
	\bigl(T_\ast-(n^2-1)\bigr)
	<0,
	\]
	we must have
	\[
	T_\ast>n^2-1.
	\]
	We may therefore write
	\begin{equation}\label{eq:epsilon-definition}
		T_\ast=n^2-1+\varepsilon,
		\qquad
		0<\varepsilon\le1.
	\end{equation}
	
We next show that \(\varepsilon<1\). Suppose that
	\(\varepsilon=1\). Then \(T_\ast=n^2\), and hence
	\[
	\sum_{i=1}^n\|x_i\|_2^2=n^2.
	\]
	Since every summand is at most \(n\), it follows that
	\[
	\|x_i\|_2^2=n
	\qquad (1\le i\le n).
	\]
	By the equality statement in Lemma~\ref{lem:row-polytope}, every row
	of \(X_\ast\) is a balanced sign vector. Thus
	\[
	X_\ast\in\{-1,1\}^{n\times n}
	\qquad\text{and}\qquad
	X_\ast\one=0.
	\]
	
Set
$
	P=I_n-\frac1nJ_n.
$
	Since \(X_\ast\one=0\), we have
	\[
	X_\ast P=X_\ast.
	\]
	Therefore, taking \(R=X_\ast\), we may write
	\[
	X_\ast=RP
	\qquad\text{with}\qquad
	R\in\{-1,1\}^{n\times n}.
	\]
	Lemma~\ref{lem:projected-sign} then gives
	\[
	\Phi(X_\ast)\ge0,
	\]
	contradicting the choice of \(X_\ast\). Hence
	\begin{equation}\label{eq:epsilon-strict}
		0<\varepsilon<1.
	\end{equation}
	
	We now estimate the largest eigenvalue of
	\(G_\ast|_{\one^\perp}\). Let
$
	\lambda_{\max}
	=
	\max_{1\le i\le n-1}\lambda_i
$
	and set
$
	\eta
	=
	\lambda_{\max}
	-
T_\ast/(n-1).
$
	The remaining \(n-2\) deviations from the mean
	\(T_\ast/(n-1)\) have sum \(-\eta\). By the Cauchy--Schwarz
	inequality, the sum of their squares is at least
$
	\eta^2/(n-2).
$
	It follows that
	\[
	V_\ast
	\ge
	\eta^2+\frac{\eta^2}{n-2}
	=
	\frac{n-1}{n-2}\eta^2.
	\]
	Consequently,
	\[
	\lambda_{\max}
	\le
	\frac{T_\ast}{n-1}
	+
	\sqrt{\frac{n-2}{n-1}\,V_\ast}.
	\]
	
	Since \(\Phi(X_\ast)<0\), equation
	\eqref{eq:epsilon-definition} gives
	\[
	V_\ast
	<
	\frac{n^2(n-2)}{n-1}\varepsilon.
	\]
	Therefore,
	\begin{equation}\label{eq:lambda-max}
		\lambda_{\max}
		<
		\frac{
			T_\ast+n(n-2)\sqrt{\varepsilon}
		}{n-1}.
	\end{equation}
	
	We next prove that every row of \(X_\ast\) is an extreme point of
	\(\mathcal{K}_n\). Fix \(x\in\cK_n\) such that \(x^{\mathsf T}\) is a row of
	\(X_\ast\), and suppose that \(x\) is not an extreme point of
	\(\cK_n\). Then \(x\) lies
	in the relative interior of a nontrivial line segment contained in
	\(\mathcal{K}_n\). After rescaling the direction of this segment,
	there exists a nonzero vector \(z\in\one^\perp\) such that
	\[
	x+\tau z\in\mathcal{K}_n
	\qquad (-1\le\tau\le1).
	\]
	
	For \(-1\le\tau\le1\), let \(X_\tau\) be obtained from \(X_\ast\)
	by replacing the row \(x^{\mathsf T}\) with
	\((x+\tau z)^{\mathsf T}\), and define
	\[
	\phi(\tau)=\Phi(X_\tau).
	\]
	Every \(X_\tau\) is admissible. Since \(X_\ast\) is a global
	minimizer of \(\Phi\), the function \(\phi\) has a local minimum at
	\(\tau=0\). Hence
	\begin{equation}\label{eq:second-variation-nonnegative}
		\phi''(0)\ge0.
	\end{equation}
	
	On the other hand, applying
	Lemma~\ref{lem:row-second-variation} to \(X_\ast\) gives
	\begin{align*}
		\frac14\phi''(0)
		={}&
		z^{\mathsf T}G_\ast z
		+
		\left(1-\frac{2}{n-1}\right)(x^{\mathsf T}z)^2\\
		&+
		\left(
		\|x\|_2^2
		-\frac{T_\ast}{n-1}
		-\frac{n^2(n-2)}{2(n-1)}
		\right)
		\|z\|_2^2.
	\end{align*}
	
	Since \(z\in\one^\perp\), the Rayleigh--Ritz principle \cite[p.~235, Theorem~4.2.2(c)]{HJ13} gives
	\[
	z^{\mathsf T}G_\ast z
	\le
	\lambda_{\max}\|z\|_2^2.
	\]
	Moreover,
	\[
	(x^{\mathsf T}z)^2
	\le
	\|x\|_2^2\|z\|_2^2,
	\]
	and Lemma~\ref{lem:row-polytope} gives
	\[
	\|x\|_2^2\le n.
	\]
	Since
$
	1-2/(n-1)>0,
$
	we obtain
	\begin{align*}
		\frac14\phi''(0)
		&\le
		\left[
		\lambda_{\max}
		+
		\left(2-\frac{2}{n-1}\right)\|x\|_2^2
		-\frac{T_\ast}{n-1}
		-\frac{n^2(n-2)}{2(n-1)}
		\right]
		\|z\|_2^2\\
		&\le
		\left[
		\lambda_{\max}
		+
		\frac{2n(n-2)}{n-1}
		-\frac{T_\ast}{n-1}
		-\frac{n^2(n-2)}{2(n-1)}
		\right]
		\|z\|_2^2.
	\end{align*}
	Using \eqref{eq:lambda-max}, we find that
	\begin{equation*}
		\frac14\phi''(0)
		<
		\frac{n(n-2)}{n-1}
		\left(
		2+\sqrt{\varepsilon}-\frac n2
		\right)
		\|z\|_2^2.
	\end{equation*}
	
	If \(n\ge8\), then, by \eqref{eq:epsilon-strict},
	\[
	2+\sqrt{\varepsilon}-\frac n2
	<
	3-\frac n2
	\le-1.
	\]
	If \(n=6\), then
	\[
	2+\sqrt{\varepsilon}-\frac n2
	=
	\sqrt{\varepsilon}-1
	<0.
	\]
	Thus, in every case,
	\[
	\phi''(0)<0.
	\]
	This contradicts
	\eqref{eq:second-variation-nonnegative}. We conclude that every row
	of \(X_\ast\) is an extreme point of \(\mathcal{K}_n\).
	
	It remains to identify the form of these extreme points. By
	Lemma~\ref{lem:row-polytope},
	\[
	\mathcal{K}_n=P[-1,1]^n,
	\qquad
	P=I_n-\frac1nJ_n.
	\]
	Let \(x\) be an extreme point of \(\mathcal{K}_n\). Since
	\(\mathcal{K}_n\) is a polytope, there exists a linear functional
	\(\ell\) that is uniquely maximized over \(\mathcal{K}_n\) at \(x\).
	
	The linear functional
$
	s\longmapsto\ell(Ps)
$
	attains its maximum over the cube \([-1,1]^n\) at some vertex
$
	s\in\{-1,1\}^n.
$
	Thus \(Ps\) maximizes \(\ell\) over \(\mathcal{K}_n\). By the
	uniqueness of the maximizer,
	\[
	x=Ps.
	\]
	
	Applying this argument to every row of \(X_\ast\), choose sign vectors
	\(s_1,\ldots,s_n\in\{-1,1\}^n\) such that the \(i\)-th row of
	\(X_\ast\) is \((Ps_i)^{\mathsf T}\). Let
	\(R\in\{-1,1\}^{n\times n}\) be the matrix whose \(i\)-th row is
	\(s_i^{\mathsf T}\). Since \(P=P^{\mathsf T}\), the \(i\)-th row of
	\(RP\) is
	\[
	s_i^{\mathsf T}P
	=
	(Ps_i)^{\mathsf T}.
	\]
	Therefore,
	\[
	X_\ast=RP.
	\]
	
	Lemma~\ref{lem:projected-sign} now gives
	\[
	\Phi(X_\ast)\ge0,
	\]
	contradicting \(\Phi(X_\ast)<0\). Hence \(\Phi\) is nonnegative on
	the set of admissible matrices. In particular,
	\[
	\Phi(D)\ge0
	\]
	for the matrix \(D\) in the statement, which is exactly
	\eqref{eq:frame-gap}.
\end{proof}

The next elementary lemma converts spectral variance into a lower
bound for the sum of reciprocal eigenvalues.

\begin{lem}\label{lem:reciprocal-variance}
	Let \(r\ge2\), let
	\(\lambda_1,\ldots,\lambda_r>0\), and define
	\[
	T=\sum_{i=1}^r\lambda_i,
	\qquad
	\overline{\lambda}=\frac{T}{r},
\qquad
	V
	=
	\sum_{i=1}^r
	(\lambda_i-\overline{\lambda})^2.
	\]
	Then
	\begin{equation}\label{eq:reciprocal-variance}
		\sum_{i=1}^r\frac1{\lambda_i}
		\ge
		\frac{r^2}{T}
		+
		\frac{V}{
			\overline{\lambda}^{\,2}
			\left(
			\overline{\lambda}
			+\sqrt{\frac{r-1}{r}\,V}
			\right)}.
	\end{equation}
\end{lem}

\begin{proof}
	Let
$
	\lambda_{\max}=\max_{1\le i\le r}\lambda_i
$
	and set
$
	d=\lambda_{\max}-\overline{\lambda}.
$
	The remaining \(r-1\) deviations from
	\(\overline{\lambda}\) have sum \(-d\). The
	Cauchy--Schwarz inequality therefore gives
	\[
	V
	\ge
	d^2+\frac{d^2}{r-1}
	=
	\frac{r}{r-1}d^2.
	\]
	Thus
	\begin{equation}\label{eq:largest-eigenvalue-general}
		\lambda_{\max}
		\le
		\overline{\lambda}
		+
		\sqrt{\frac{r-1}{r}\,V}.
	\end{equation}
	
	For every \(\lambda>0\), the following identity is exact:
	\[
	\frac1\lambda
	=
	\frac1{\overline{\lambda}}
	-
	\frac{\lambda-\overline{\lambda}}
	{\overline{\lambda}^{\,2}}
	+
	\frac{(\lambda-\overline{\lambda})^2}
	{\overline{\lambda}^{\,2}\lambda}.
	\]
	Summing this identity over \(i\) and using
	\[
	\sum_{i=1}^r
	(\lambda_i-\overline{\lambda})=0,
	\]
	we obtain
	\begin{equation}\label{eq:reciprocal-identity}
		\sum_{i=1}^r\frac1{\lambda_i}
		=
		\frac{r^2}{T}
		+
		\sum_{i=1}^r
		\frac{
			(\lambda_i-\overline{\lambda})^2
		}{
			\overline{\lambda}^{\,2}\lambda_i
		}.
	\end{equation}
	
	Since \(\lambda_i\le\lambda_{\max}\),
	\[
	\sum_{i=1}^r
	\frac{
		(\lambda_i-\overline{\lambda})^2
	}{
		\overline{\lambda}^{\,2}\lambda_i
	}
	\ge
	\frac{V}{
		\overline{\lambda}^{\,2}\lambda_{\max}
	}.
	\]
	Combining this estimate with
	\eqref{eq:largest-eigenvalue-general} and
	\eqref{eq:reciprocal-identity} proves
	\eqref{eq:reciprocal-variance}.
\end{proof}

We can now complete the proof for the case  \(n\ge6\) in Theorem~\ref{thm:centered}.

\begin{prop}\label{prop:spectral-conversion}
	The conclusion of Theorem~\ref{thm:centered} holds for all even \(n\ge 6\).
\end{prop}

\begin{proof}
	Let \(D\) satisfy the hypotheses of
	Theorem~\ref{thm:centered}. Set
$
	r=n-1,
	G=D^{\mathsf T}D,
$
	and let \(\lambda_1,\ldots,\lambda_r>0\) be the nonzero
	eigenvalues of \(G\).
	Since \(\rank D=r\),
$
	\|D^\dagger\|_F^2
	=
	\sum_{i=1}^r1/\lambda_i.
$
	Define
	\[
	T=\sum_{i=1}^r\lambda_i
	=
	\tr G
	=
	\|D\|_F^2\qquad
\text{and}\qquad
	V
	=
	\sum_{i=1}^r
	\left(
	\lambda_i-\frac{T}{r}
	\right)^2.
	\]
	
	By Lemma~\ref{lem:row-polytope},
	\[
	\|d_i\|_2^2\le n
	\qquad (1\le i\le n),
	\]
	and hence
	\[
	T\le n^2.
	\]
	
	Suppose first that
$
	T\le n^2-1=r(n+1).
$
	The arithmetic--harmonic mean inequality gives
	\[
	\|D^\dagger\|_F^2
	=
	\sum_{i=1}^r\frac1{\lambda_i}
	\ge
	\frac{r^2}{T}
	\ge
	\frac{r}{n+1}.
	\]
	
	It remains to consider
$
	T>n^2-1.
$
	Write
$
		T=n^2-1+\varepsilon$, where $
		0<\varepsilon\le1.
$
Since \(D\one=0\) and \(\rank D=r=n-1\), the spectrum of \(G\) is
\[
0,\lambda_1,\ldots,\lambda_r.
\]
Hence
\[
V
=
\tr(G^2)-\frac{T^2}{n-1}.
\]
Moreover, the hypotheses on the rows of \(D\) imply that they belong
to \(\cK_n\). Therefore Lemma~\ref{lem:frame-gap} gives
	\begin{equation}\label{eq:variance-lower}
		V
		\ge
		\frac{n^2(n-2)}{r}\varepsilon.
	\end{equation}
		Set
$
	\overline{\lambda}=T/r.
$
	Lemma~\ref{lem:reciprocal-variance} gives
	\begin{equation}\label{eq:reciprocal-lower}
		\|D^\dagger\|_F^2
		\ge
		\frac{r^2}{T}
		+
		\frac{V}{
			\overline{\lambda}^{\,2}
			\left(
			\overline{\lambda}
			+\sqrt{\frac{r-1}{r}\,V}
			\right)}.
	\end{equation}
	For fixed \(\overline{\lambda}>0\), the function
	\[
	v\longmapsto
	\frac{v}{
		\overline{\lambda}
		+\sqrt{\frac{r-1}{r}\,v}
	}
	\]
	is increasing on \([0,\infty)\). Indeed, for \(v>0\), its derivative
	is
	\[
	\frac{
		\overline{\lambda}
		+\frac12\sqrt{\frac{r-1}{r}\,v}
	}{
		\left(
		\overline{\lambda}
		+\sqrt{\frac{r-1}{r}\,v}
		\right)^2
	}
	>0.
	\]
	We may therefore substitute the lower bound
	\eqref{eq:variance-lower} into the second term of
	\eqref{eq:reciprocal-lower}. Since \(r=n-1\),
	\[
	\sqrt{
		\frac{r-1}{r}
		\cdot
		\frac{n^2(n-2)}{r}\varepsilon
	}
	=
	\frac{n(n-2)}{r}\sqrt{\varepsilon}.
	\]
	It follows that
	\begin{equation}\label{eq:correction-lower}
		\frac{V}{
			\overline{\lambda}^{\,2}
			\left(
			\overline{\lambda}
			+\sqrt{\frac{r-1}{r}\,V}
			\right)}
		\ge
		\frac{
			n^2(n-2)\varepsilon r^2
		}{
			T^2
			\left(
			T+n(n-2)\sqrt{\varepsilon}
			\right)
		}.
	\end{equation}
	
	The difference between the desired bound and the first term in
	\eqref{eq:reciprocal-lower} is
	\begin{equation}\label{eq:reciprocal-deficit}
		\frac{r}{n+1}-\frac{r^2}{T}
		=
		\frac{r\varepsilon}{(n+1)T}.
	\end{equation}
	Therefore, it suffices to prove that
	\begin{equation}\label{eq:scalar-spectral}
		(n+1)n^2(n-2)r
		\ge
		T\left(
		T+n(n-2)\sqrt{\varepsilon}
		\right).
	\end{equation}
	Since \(T\le n^2\) and \(\varepsilon\le1\),
	\begin{align*}
		T\left(
		T+n(n-2)\sqrt{\varepsilon}
		\right)
		&\le
		n^2\bigl(n^2+n(n-2)\bigr)\\
		&=
		2n^3(n-1)\\
		&=
		2n^3r.
	\end{align*}
	On the other hand,
	\[
	(n+1)n^2(n-2)r
	\ge
	2n^3r,
	\]
	because
	\[
	(n+1)(n-2)\ge2n
	\qquad (n\ge4).
	\]
	This proves \eqref{eq:scalar-spectral}. Combining
	\eqref{eq:reciprocal-lower},
	\eqref{eq:correction-lower}, and
	\eqref{eq:reciprocal-deficit}, we obtain
	\[
	\|D^\dagger\|_F^2
	\ge
	\frac{r}{n+1}
	=
	\frac{n-1}{n+1}.
	\]
\end{proof}

\begin{proof}[Proof of Theorem~\ref{thm:centered}]
	For \(n=4\), the conclusion follows from
Proposition~\ref{prop:n-four}. For every even \(n\ge6\), it follows from
Proposition~\ref{prop:spectral-conversion}.
\end{proof}

\section{The even case}\label{s:4}
In this section, we prove the even-dimensional case of
Theorem~\ref{thm:main}.

With the centered estimate in Theorem~\ref{thm:centered} available, it remains to recover the
direction annihilated by row centering.  The following exact
decomposition isolates that one-dimensional contribution.

\begin{lem}\label{lem:inverse-decomposition}
Let \(A\in\R^{n\times n}\) be nonsingular, put
\[
 P=I_n-\frac1nJ_n,
 \qquad
 \mu=\frac1nA\one,
 \qquad
 D=2AP,
\]
and let \(u\) be a unit vector spanning \(\ker D^{\mathsf T}\).  Define
$
 t=u^{\mathsf T}\mu,
 w=\mu-tu.
$
Then \(t\ne0\), and
\begin{equation}\label{eq:inverse-decomposition}
 A^{-1}
 =2D^\dagger+
 \left(\frac1{nt}\one-\frac2tD^\dagger w\right)u^{\mathsf T}.
\end{equation}
Moreover,
\begin{equation}\label{eq:norm-decomposition}
 \|A^{-1}\|_F^2
 =4\|D^\dagger\|_F^2+\frac1{nt^2}
 +\frac4{t^2}\|D^\dagger w\|_2^2.
\end{equation}
\end{lem}

\begin{proof}
Since \(A\) is invertible and \(P\) has rank \(n-1\), \(D\) has rank
\(n-1\), \(\ker D=\Span\{\one\}\), and
\(\operatorname{range}D=u^\perp\).  If \(t=0\), then
\(\mu\in u^\perp=\operatorname{range}D\), and every column of
$
 A=1/2D+\mu\one^{\mathsf T}
$
would lie in \(\operatorname{range}D\), contradicting the
nonsingularity of \(A\).  Thus \(t\ne0\), and
\(w\in\operatorname{range}D\).

The standard Moore--Penrose identities in the present rank-one defect
setting are
\[
 D^\dagger D=P,
 \qquad
 DD^\dagger=I_n-uu^{\mathsf T},
 \qquad
 D^\dagger u=0.
\]
Multiplying the right-hand side of \eqref{eq:inverse-decomposition} by
\(A=D/2+\mu\one^{\mathsf T}\), and using
\(D^\dagger\mu=D^\dagger w\), gives
\[
 P+2D^\dagger w\one^{\mathsf T}
 +\frac1nJ_n-2D^\dagger w\one^{\mathsf T}=I_n.
\]
Thus the displayed matrix is a left inverse of \(A\). Since \(A\) is
square and nonsingular, this left inverse equals \(A^{-1}\), proving
\eqref{eq:inverse-decomposition}.

Finally,
\[
 \operatorname{range}(D^\dagger)
 =\operatorname{range}(D^{\mathsf T})
 =(\ker D)^\perp
 =\one^\perp,
\]
so \(D^\dagger w\perp\one\).  Also \(D^\dagger u=0\), which makes
\(D^\dagger\) Frobenius-orthogonal to every matrix of the form
\(zu^{\mathsf T}\).  More explicitly,
\[
 \langle D^\dagger,zu^{\mathsf T}\rangle_F
 =(D^\dagger u)^{\mathsf T}z=0,
 \qquad
 \langle\one u^{\mathsf T},(D^\dagger w)u^{\mathsf T}\rangle_F
 =\one^{\mathsf T}D^\dagger w=0.
\]
Thus the three terms in \eqref{eq:inverse-decomposition} are pairwise
Frobenius-orthogonal, and \eqref{eq:norm-decomposition} follows.
\end{proof}

The next identity measures simultaneously the loss under centering and
the distance from the binary cube.

\begin{lem}\label{lem:centered-energy}
Assume that \(0\le a_{ij}\le1\), and use the notation of
Lemma~\ref{lem:inverse-decomposition}.  Put
\[
 T=\|D\|_F^2,
 \qquad
 E=\sum_{i,j}a_{ij}(1-a_{ij}),
 \qquad
 q=\left\|\mu-\frac12\one\right\|_2^2,
 \qquad
 x=2\sqrt{nq}.
\]
Then
\begin{equation}\label{eq:energy-identity}
 n^2-T=x^2+4E.
\end{equation}
Furthermore,
\begin{equation}\label{eq:mu-bound}
 \|\mu\|_2\le\frac{n+x}{2\sqrt n}.
\end{equation}
\end{lem}

\begin{proof}
Since \(D=2(A-\mu\one^{\mathsf T})\), rowwise expansion gives
\[
 T=4\sum_{i,j}a_{ij}^2-4n\|\mu\|_2^2.
\]
On the other hand,
\begin{align*}
 4nq+4E
 &=4n\|\mu\|_2^2-4n\one^{\mathsf T}\mu+n^2
   +4\sum_{i,j}a_{ij}-4\sum_{i,j}a_{ij}^2\\
 &=n^2-T,
\end{align*}
because \(\sum_{i,j}a_{ij}=n\one^{\mathsf T}\mu\).  This proves
\eqref{eq:energy-identity}.  The triangle inequality gives
\[
 \|\mu\|_2
 \le\frac{\sqrt n}{2}+\sqrt q
 =\frac{n+x}{2\sqrt n},
\]
which is \eqref{eq:mu-bound}.
\end{proof}

We can now prove strictness in every even dimension at least four.  The
two cases below are divided at \(x=1\), exactly where the centered
estimate and the elementary singular-value estimate meet.

\begin{prop}\label{prop:even-large}
Let \(n\ge4\) be even.  Under the hypotheses of
Theorem~\ref{thm:main},
\[
 \|A^{-1}\|_F>\frac{2n}{n+1}.
\]
\end{prop}

\begin{proof}
Use the notation of Lemmas~\ref{lem:inverse-decomposition}
and~\ref{lem:centered-energy}.  Since \(P\one=0\), we have \(D\one=0\). If
\(a_i^{\mathsf T}\) and \(d_i^{\mathsf T}\) are the \(i\)-th rows of
\(A\) and \(D\), respectively, then subtracting the row mean does not
change oscillation, and hence
\[
\osc(d_i)=2\osc(a_i)\le2.
\]
Moreover, since \(A\) is nonsingular and \(\rank P=n-1\),
\[
\rank D=\rank(2AP)=n-1.
\]
Therefore Theorem~\ref{thm:centered} applies.

Suppose first that \(x<1\).  By \eqref{eq:norm-decomposition},
Theorem~\ref{thm:centered}, \(|t|\le\|\mu\|_2\), and
\eqref{eq:mu-bound},
\begin{align*}
 \|A^{-1}\|_F^2
 &\ge4\frac{n-1}{n+1}+\frac1{n\|\mu\|_2^2}\\
 &>4\frac{n-1}{n+1}+\frac4{(n+1)^2}
 =\frac{4n^2}{(n+1)^2}.
\end{align*}

Now assume \(x\ge1\).  Since \(E\ge0\) and
$
 T=\|D\|_F^2>0,
$
the identity
\[
 T=n^2-x^2-4E
\]
implies \(1\le x<n\).  The arithmetic--harmonic mean
inequality applied to the \(n-1\) nonzero squared singular values of
\(D\) gives
\[
 \|D^\dagger\|_F^2\ge\frac{(n-1)^2}{T}
 \ge\frac{(n-1)^2}{n^2-x^2}.
\]
Together with \eqref{eq:norm-decomposition}, \(|t|\le\|\mu\|_2\), and
\eqref{eq:mu-bound}, this yields
\begin{equation}\label{eq:Lx}
 \|A^{-1}\|_F^2
 \ge L(x):=
 \frac{4(n-1)^2}{n^2-x^2}+\frac4{(n+x)^2}.
\end{equation}
For \(1\le x<n\), direct differentiation shows that
\[
 L'(x)\ge0
 \quad\Longleftrightarrow\quad
 (n-1)^2x(n+x)\ge(n-x)^2.
\]
The last inequality is strict because its left-hand side is at least
\((n-1)^2(n+1)\), while its right-hand side is at most
\((n-1)^2\).  Therefore
\begin{equation}\label{eq:L-one}
 L(x)\ge L(1)
 =\frac{4(n-1)}{n+1}+\frac4{(n+1)^2}
 =\frac{4n^2}{(n+1)^2}.
\end{equation}

It remains to rule out equality in the second case.  If equality held
in the desired estimate, then equality would have to hold at every
step in \eqref{eq:Lx}--\eqref{eq:L-one}.  Since \(L\) is strictly
increasing on \([1,n)\), this first gives \(x=1\).  Equality in
\[
 T\le n^2-x^2
\]
then gives \(E=0\) by \eqref{eq:energy-identity}; hence every entry of
\(A\) is either \(0\) or \(1\).  Equality in \eqref{eq:mu-bound} forces
\(\mu-\one/2\) to be a nonnegative multiple of \(\one\).  Since
\(q=1/(4n)\), it follows that
\[
 \mu=\frac{n+1}{2n}\one.
\]
Thus every row of the binary matrix \(A\) would contain
\((n+1)/2\) ones, which is impossible because \(n\) is even.  Hence
equality cannot occur.
\end{proof}

The case \(n=2\) is handled directly.

\begin{lem}\label{lem:n-two}
If \(A\in[0,1]^{2\times2}\) is nonsingular, then
\[
 \|A^{-1}\|_F^2\ge2>\frac{16}{9}.
\]
\end{lem}

\begin{proof}
Put \(F=\|A\|_F^2\).  For a \(2\times2\) matrix,
\[
 \|A^{-1}\|_F^2=\frac{F}{(\det A)^2}.
\]
If \(F\le2\), Hadamard's determinant inequality \cite[p.~505, Theorem~7.8.1]{HJ13} and the
arithmetic--geometric mean inequality applied to the squared row
norms give
\[
 (\det A)^2\le\frac{F^2}{4}\le\frac F2.
\]
If \(F\ge2\), then \(|\det A|\le1\), since
\(|a_{11}a_{22}-a_{12}a_{21}|\le1\), and again
\((\det A)^2\le F/2\).  The assertion follows.
\end{proof}

Next, we prove Theorem~\ref{thm:main}.

\begin{proof}[Proof of Theorem~\ref{thm:main}]
	For odd \(n\), including \(n=1\), the inequality and its equality
	statement follow from Proposition~\ref{prop:odd}. For even \(n\ge4\),
	Proposition~\ref{prop:even-large} gives strict inequality, while
	Lemma~\ref{lem:n-two} handles \(n=2\). Finally, \eqref{eq:Svalue}
	shows that every \(S\)-matrix attains equality. These statements
	together prove the theorem.
\end{proof}

\begin{rem}\label{rem:equality-existence}
For even \(n\ge2\), the proof gives strict inequality.  For odd \(n\),
equality is attained exactly when a Hadamard matrix of order \(n+1\)
exists.  Theorem~\ref{thm:main} therefore characterizes all possible equality
cases without asserting existence in the unresolved Hadamard orders.
\end{rem}

\enlargethispage{3\baselineskip}

\end{document}